\documentclass[a4paper,10pt]{amsart}

\usepackage[latin1]{inputenc}
\usepackage{amsmath, amsthm, amssymb}
\usepackage{amscd}
\usepackage[dvips]{graphicx}
\usepackage[all]{xy}
\usepackage{enumerate}
\usepackage{hyperref}

\newtheorem{theorem}{Theorem}
\newtheorem{corollary}[theorem]{Corollary}
\newtheorem{proposition}[theorem]{Proposition}
\newtheorem{definition}{Definition}
\newtheorem{lemma}[theorem]{Lemma}

\newtheorem*{theorem*}{Theorem}
\newtheorem*{proposition*}{Proposition}
\newtheorem*{definition*}{Definition}
\newtheorem*{lemma*}{Lemma}
\newtheorem*{claim*}{Claim}
\newtheorem*{corollary*}{Corollary}

\theoremstyle{remark}
\newtheorem{rem}[theorem]{Remark}
\newtheorem*{rem*}{Remark}

\newcommand{\wt}[1]{\widetilde{#1}}
\newcommand{\R}{\mathbb R}
\newcommand{\Z}{\mathbb Z}
\newcommand{\T}{\mathbb T}
\newcommand{\Hyp}{\mathbb H}

\newcommand{\M}{\widetilde{M}}
\newcommand\ada{A \wedge dA^{n-1}}
\newcommand\eps{\varepsilon}

\renewcommand{\S}{\mathbb{S}}

\DeclareMathOperator{\Ker}{Ker}
\DeclareMathOperator{\voleucl}{\mathrm{vol}_{\mathrm{Eucl}}}

\DeclareMathOperator{\vol}{\mathrm{vol}}

\usepackage{color}

\title[Deformations of spectrum that preserve the length spectrum]{On deformations of the spectrum of a Finsler--Laplacian that preserve the length spectrum}
\author{Thomas Barthelm\'e}

\begin{document}
\begin{abstract}
The main result of this article is the construction of non-reversible Finsler metrics in negative curvature such that $4\lambda_1 > h^2$, where $\lambda_1$ is the bottom of the $L^2$-spectrum of a previously defined Finsler--Laplacian and $h$ the topological entropy of the flow. This gives a counter-example to a classical inequality in Riemannian geometry. 
We also show that the spectrum of that Finsler--Laplacian can detect changes in the Finsler metric that the marked length spectrum cannot.
\end{abstract}

 \maketitle

Finsler metrics have a long history of producing quite different results from what one might expect from Riemannian metrics. Among the general classes of Finsler metrics from which surprises can arise are non-reversible Finsler metrics. Non-reversible Finsler metrics are defined by considering norms which are not symmetric with respect to $0$, or in other words, such that their unit balls in each tangent space are convex sets that contain, but are not centered at, the origin.
 One of the most striking surprise that arose from non-reversible Finsler metrics was the construction by Katok \cite{Katok:KZ_metric} in 1973 of a metric on the sphere with only 2 periodic geodesics (these metrics are now called Katok-Ziller metrics as they have been thoroughly studied by Ziller in \cite{Ziller:GKE}). The Katok-Ziller metrics turned out to be Randers metric, i.e., metrics of the form $F= \sqrt{g} +\beta$ where $g$ is a Riemannian metric and $\beta$ is a one-form.

We are interested in this article in the Finsler--Laplacian spectrum of Randers metrics, or more generally Finsler metrics of the form $F= \bar F +\beta$ where $\bar{F}$ is a reversible Finsler metric and $\beta$ a one-form. The operator we consider is the Finsler--Laplacian introduced in \cite{moi:these,moi:natural_finsler_laplace}. Since this operator was thought of by Jean-Pierre Bourguignon and Patrick Foulon, who then suggested it to me, we will henceforth call this operator the BF-Laplacian (we recall its construction in Section \ref{sec:background} below). 

While considering this operator, we already had some surprising results in the non-reversible case:
In \cite{BarthelmeColbois}, Colbois and myself showed that, for any surface $S$ and any \emph{reversible} Finsler metric $\bar{F}$, there exists a \emph{uniform} constant $K$ (depending only on the topology of $S$) such that $\lambda_1(\bar{F}) \vol\left(S, \bar{F} \right) \leq K$. This result is just a generalization of a classical Riemannian result \cite{Korevaar:Upper_bounds}.
But we also proved that, for any $C>0$ and any surface $S$, there exists a Randers metric $F$ such that $\lambda_1(F) \vol\left(S, F \right) \geq C$. So, allowing a metric to be non-reversible can yield examples of metrics with a $\lambda_1$ much bigger than it should be.

We will construct here examples of non-reversible metrics that yield two more surprises. The first with respect to a presumed link between marked length spectrum and the spectrum of the Laplacian and the second with respect to the link between the bottom of the spectrum and the topological entropy of the geodesic flow.

The length spectrum of a metric is defined as the set of lengths of closed geodesics counted with multiplicity. Two manifolds are said to have the same \emph{marked} length spectrum if there is an isomorphism of their fundamental group such that corresponding free homotopy classes contain closed geodesics of the same length. The link between Laplacian spectrum and length, or marked length, spectrum has been intensively studied in Riemannian geometry. Generically, the length spectrum of a Riemannian manifold is determined by the Laplacian spectrum (Colin de Verdi\`ere \cite{CdV:spectre_et_longueurs}). In some specific cases, the notions of marked length spectrum and Laplacian spectrum are in fact equivalent in the sense that one determines the other and vice versa. Among the manifolds that verifies this are for instance flat tori (see for instance \cite{Gordon:when_you_can't_hear_the_shape}), manifolds of negative curvature (Otal \cite{Otal:spectre_marque} 
and Croke \cite{Croke:rigidity} since in that case equality of the marked length spectrum implies isometry), and some types of nilmanifolds (see \cite{Eberlein:2_step}, and it is in fact conjectured to be true for all nilmanifolds \cite{Gornet}).

Examples of Riemannian manifolds with the same length spectrum but not isospectral exists however, but they are quite exceptional. One such example is comparing two different Zoll surfaces (i.e., a metric on the sphere such that all its geodesics are closed and of length $2\pi$, see \cite{Guillemin}).

Non-reversible Finsler metrics give a very contrasted picture: for \emph{any} (reversible) metric on \emph{any} manifold, we can construct a non-reversible metric with the same marked length spectrum and different spectra:
\begin{theorem} \label{thm:same_length_spectrum_higher_spectrum}
 Let $\bar{F}$ be a Finsler metric on a manifold $M$, we denote by $\bar{F}^{\ast}$ the dual metric. Let $F = \bar{F} + \beta$, where $\beta$ is an \emph{exact} $1$-form on $M$, not identically zero, such that $\bar{F}^{\ast}(\beta)<1$. 
Then $\bar{F}$ and $F$ have the same marked length spectrum and the same volume, but, for each $k$, $\lambda_k(F)>\lambda_k(\bar{F})$.
\end{theorem}
Note that the condition on the norm of $\beta$ is only there to insure that the metric $F$ is still a Finsler metric.

Saying that this result is really surprising might be a bit of a stretch. Indeed, there exist infinite-dimensional families of Finsler metrics that share the same marked length spectrum, so finding some metrics with different spectra should not be too hard. But on the other hand, infinitely many Finsler metrics should also share the same BF-Laplacian (see \cite{moi:these,moi:natural_finsler_laplace}), which makes the existence of the above examples not completely obvious.

Moreover, the main interest of this result is what it suggests about the BF-Laplacian: this type of transformation of a reversible metric by an exact form does not change the metric, or the geodesic flow a lot. Indeed, the new geodesic flow is a time change of the old that do not change the length of any closed geodesic. In fact such a time-change is a trivial time change in the terminology of \cite{KatokHassel}, i.e., it is a time change such that the two flows are smoothly conjugate (and this is all due to the fact that $\beta$ is taken to be exact, see Lemma \ref{lem:canonical_time_change}).
So the length spectrum is not subtle enough to pick up this change, nor is the dynamics of the geodesic flow. But what the above result shows is that the BF-Laplacian do detect such variations, which could make it a more powerful tool in some situations.

If we work a bit more, we can obtain some even more surprising examples:
\begin{theorem}\label{thm:same_length_exploding_lambda_1}
 Let $g_0$ be the flat metric on the $2$-torus $\R^2 / \Z^2$. Let $F_{\eps,t}= \sqrt{g_0} +tdh_{\eps}$ be a Randers metric, where $h_{\eps}$ is a well chosen function such that, almost everywhere, $\nabla h_{\eps}$ tends to a unit vector of irrational slope.
Then, for all $\eps,t$, $(\mathbb{T}^2, F_{\eps,t})$ have the same volume, the same geodesic flow up to a (trivial) time-change and the same marked length spectrum as $(\mathbb{T}^2, g_0)$, but
\begin{equation*}
 \lim_{(\eps, t) \rightarrow (0,1)} \lambda_1(F_{\eps,t}) = +\infty.
\end{equation*}
\end{theorem}
The family of functions $h_{\eps}$ are given explicitly in Section \ref{sec:fixed_length_exploding_lambda}.
This result is an improvement on the example on the torus constructed in \cite{BarthelmeColbois}. First of all because this new example preserves the marked length spectrum. But also, in \cite{BarthelmeColbois}, we had to modify the Riemannian part of our Randers metric in order to build a big eigenvalue, whereas this example shows that the Riemannian part can be fixed. It seems very likely that one could build ad hoc examples of a family of Randers metric with a first eigenvalue tending to infinity on any manifold and with any fixed Riemannian part. However the construction and the proof is much easier in the torus case and I did not investigate more the general case.

We will now see what the same type of construction can yield in terms of the relation between the bottom of the $L^2$-spectrum and the topological entropy of the geodesic flow. In the following, $M$ is a closed manifold equipped with a Finsler metric $F$, $\wt M$ is the universal cover of $M$ and $\wt F$ the lifted metric. We denote by $\wt \Delta$ the BF-Laplacian of $\wt F$ and call $\lambda_1(\wt F)$ the bottom of the $L^2$-spectrum of $\wt \Delta$ (see Section \ref{sec:negatively_curved} for more details).

A classical result in Riemannian geometry is the inequality $4\lambda_1(\wt g) \leq h(g)^2$, where $h(g)$ is the topological entropy of the geodesic flow. Moreover, a very interesting rigidity phenomenon takes place for Riemannian metrics: if $4\lambda_1(\wt g) = h(g)^2$, then $(M,g)$ is a Riemannian symmetric space (see \cite{Katok:4_appl,Ledrappier:Harm_measures} for the surface case and \cite{BessonCourtoisGallot} in higher dimension). For quite some time, I have been hoping to prove that the inequality $4\lambda_1(\wt F) \leq h^2$ still holds for the BF-Laplacian.

 In the Finsler setting, it is very easy to show that $4\lambda_1(\wt F) \leq n h(F)^2$, where $n$ is the dimension of $M$ (see Proposition \ref{prop:weak_inequality}). In \cite{BCCV}, we also proved that the sharper inequality $4\lambda_1(\wt F) \leq h(F)^2$ does hold in some Finsler cases. However, it turns out, to my surprise, that the sharp inequality does not hold in general:
\begin{theorem} \label{thm:lambda_1_bigger_h}
 There exist examples of negatively curved, non-reversible Finsler metrics such that $4\lambda_1(\wt F) > h(F)^2$. More precisely, let $\sqrt{g}$ be an hyperbolic metric on a manifold $M$. Let $\beta$ be an \emph{exact} $1$-form on $M$ such that $\lVert \beta \rVert_{g^{\ast}} <1$. Set $F = \sqrt{g} + \beta$. If $\beta$ have only isolated zeros, then,
\[
 4\lambda_1(\wt F) > h(F)^2 = (n-1)^2.
\]
\end{theorem}
Once more, the condition $\lVert \beta \rVert_{g^{\ast}} <1$ is only there to ensure that the metric $F$ is Finsler. It also seems reasonable to expect that the condition $\beta$ have isolated zeros can be removed,
 but one would need to do more than a trivial modification of the proof we give.

The construction of Theorem \ref{thm:same_length_exploding_lambda_1} relies on the following fact: If $\bar F$ is a \emph{reversible} Finsler metric and $\beta$ a $1$-form, then the spectrum of $F =\bar F + \beta$ is strictly greater than the spectrum of $\bar F$. In fact, if two of their eigenvalues are equal, then $\beta$ has to be zero.
\begin{proposition} \label{prop:spectrum_of_Randers_Finsler}  
Let $\bar F$ be a reversible Finsler metric on a \emph{closed} manifold $M$, and let $\beta$ be a, non-identically zero, $1$-form on $M$ such that $F^*(\beta) <1$. Let $F= \bar F +\beta$. Then $\vol(M,F)=\vol(M,\bar F)$, and, for all $k\geq 1$,
\[
 \lambda_k(F) > \lambda_k(\bar F).
\]
\end{proposition}
This result was proven for the first eigenvalue $\lambda_1(F)$ and when the metric $F$ is Randers by He and Zheng \cite{HeZheng}. We provide here a general and coordinate-free proof.

Obtaining a strict inequality in Theorem \ref{thm:lambda_1_bigger_h} is however much more involved than in Proposition \ref{prop:spectrum_of_Randers_Finsler}. This is due to the fact that $\lambda_1(\wt F)$ is in general not an eigenvalue, but just the bottom of the spectrum. We can nevertheless prove the following
\begin{proposition} \label{prop:lambda_1_negatively_curved}
 Let $\bar F$ be a reversible Finsler metric on a closed manifold $M$. Let $\beta \colon M \rightarrow T^* M$ be a $1$-form on $M$ such that $\bar F^{\ast} (\beta) <1$, and $F = \bar F + \beta$. Let $\wt{\bar F}$ and $\wt F$ be the lifts to the universal cover of $M$, then
\[
 \lambda_1(\wt F) \geq \lambda_1 (\wt{\bar F}).
\]

Moreover, if $\bar F = \sqrt{g}$ is a Riemannian metric of negative curvature and $\beta$ is a $1$-form with isolated zeros, then 
\[
 \lambda_1(\wt F) > \lambda_1 (\wt{g}).
\]
\end{proposition}

Let us finish this introduction with a conjecture. Proposition \ref{prop:spectrum_of_Randers_Finsler} suggests the following problem: Let $F$ be a non-reversible Finsler metric and $\bar F$ its symmetrization, i.e., $\bar F = (F + F \circ s)/2$, where $s\colon TM \rightarrow TM$ is defined by $s(x,v)=(x,-v)$.\\
 \emph{Is the BF-Laplacian spectrum of $F$ above the spectrum of $\bar F$?}

 I suspect that this is the case, and that one can probably prove it by following the general idea of the proof of Proposition \ref{prop:spectrum_of_Randers_Finsler}. Unfortunately, the necessary computations in the general case are much more involved, and this remains a conjecture for the time being. 

Notice that this result, if true, would be similar in nature to a consequence of the Brunn--Minkowski inequality that says that the Holmes--Thompson volume of a Finsler metric is less than the Holmes--Thompson volume of its symmetrization (see, for instance, \cite{BonnesenFenchel} or \cite{Thompson_book}).

This article is organized as follows. In Section \ref{sec:background}, we recall the construction of the BF-Laplacian and how one can obtain its spectrum. In Section \ref{sec:Randers_increase_symbol}, we prove Proposition \ref{prop:spectrum_of_Randers_Finsler} and deduce Theorem \ref{thm:same_length_spectrum_higher_spectrum}. In Section \ref{sec:family_with_exploding_lambda1}, we discuss several ways of constructing families of Randers metric with a fixed Riemannian part and with unbounded $\lambda_1$ and prove Theorem \ref{thm:same_length_exploding_lambda_1}. Finally, in Section \ref{sec:negatively_curved} we consider the case of the $L^2$-spectrum for negatively curved metrics and prove Theorem \ref{thm:lambda_1_bigger_h}.

\subsection*{Acknowledgment} I would like to thank Chris Judge for a useful remark that removed an unnecessary hypothesis for Theorem \ref{thm:same_length_spectrum_higher_spectrum}.

\section{Background} \label{sec:background}
We start by recalling the definition of the BF-Laplacian and other related objects. For a more complete exposition, see \cite{moi:these}  or \cite{moi:natural_finsler_laplace}.
We will be using the following definition of Finsler metric:
\begin{definition} \label{def:finsler_metric}
Let $M$ be a manifold. A Finsler metric on $M$\ is a continuous function ${F \colon TM \rightarrow \R^+}$ that is:
\begin{enumerate}
  \item $C^{2}$\ except on the zero section,
  \item positively homogeneous, i.e., $F(x,\lambda v)=\lambda F(x,v)$\ for any $\lambda>0$,
  \item positive-definite, i.e., $F(x,v)\geq0$\ with equality if and only if $v=0$,
  \item strongly convex, i.e., $ \left(\dfrac{\partial^2 F^2}{\partial v_i \partial v_j}\right)_{i,j}$ is positive-definite.
 \end{enumerate}
\end{definition}
A Finsler metric is said to be \emph{reversible} if $F(x, -v ) = F(x,v)$\ for any $(x,v)\in TM$. We denote by $F^*$ the dual metric of $F$, it can be defined by
\[
 F^*(x,l)= \sup\{ l(v) \mid F(x,v)=1 \}.
\]

Let $HM$\ be the homogenized bundle, i.e., $HM := \left(TM \smallsetminus \{0\} \right) / \R^+$. We denote by $\pi \colon HM \rightarrow M$\ the canonical projection and by $VHM = \Ker d\pi \subset THM$ the vertical bundle. We say that a vector field $Y$ on $HM$ is \emph{vertical} if it lands in $VHM$.

The Hilbert form $A$ is a $1$-form on $HM$ defined, for $(x,\xi) \in HM$, and $Z \in T_{(x,\xi)}HM$, by
\begin{equation*}
 A_{(x,\xi)}(Z) := \lim_{\eps \rightarrow 0} \frac{F\left(x, v + \eps d\pi(Z) \right) - F\left( x,v \right)}{\eps},
\end{equation*}
where $v \in T_xM$ is a vector that projects to the direction $\xi$. That is $r(x,v) = (x,\xi)$, where $r \colon TM\smallsetminus \{0\}  \rightarrow HM$.
The Hilbert form is a contact form, i.e., if $n$ is the dimension of $M$, then $\ada$ is a volume form on $HM$. Moreover, if $X$ denotes the geodesic vector field of $F$, then $x$ is the Reeb field of $A$. That is, $X \colon HM \rightarrow THM$ is the unique vector field such that  
\begin{equation*}
\left\{ 
\begin{aligned}
  A(X) &= 1 \\
 i_X dA &= 0  \, .
 \end{aligned} \right.
\end{equation*}

In order to define the BF-Laplacian, we first split the contact volume $\ada$ into a volume form on the manifold $M$ and an angle form: There exist a unique volume form $\Omega^F$\ on $M$\ and a $(n-1)$-form $\alpha^F$\ on $HM$, never zero on $VHM$, such that
\begin{equation*}
  \alpha^{F} \wedge \pi^{\ast}\Omega^F =  A\wedge dA^{n-1}, 
\end{equation*}
and, for all $x\in M$, 
\begin{equation*}
 \int_{H_xM} \alpha^F =  \voleucl(\S^{n-1})\, .
\end{equation*}
Note that $\alpha^F$ is not technically unique as a $(n-1)$-form, but its integration along a Borel set in a fiber $H_xM$ is. So it is unique only as an angle measure, but this is all we need.
Note also that the volume form $(n-1)!^{-1} \Omega^F$ is the Holmes--Thompson volume form, but since the factor $(n-1)!$ does not play any role in all that we do, we just say that $\Omega^F$ is the Holmes--Thompson volume.

The Bourguignon--Foulon--Laplacian of a function is then obtained as the average with respect to $\alpha^F$ of the second derivatives in every directions:
\begin{definition}
\label{def:delta}
 For $f \in C^2(M)$, the BF-Laplacian is the operator $\Delta^F$\ defined by, for any $x \in M$,
 $$
 \Delta^F f (x) = \frac{n}{\voleucl \left(\mathbb{S}^{n-1}\right) }\int_{H_xM} L_X ^2 (\pi^{\ast} f ) \alpha^F,
 $$
where $L_X$ denotes the Lie derivative of $X$.
\end{definition}

When the manifold $M$ is compact, the BF-Laplacian admits a discrete, unbounded spectrum $0=\lambda_0 < \lambda_1 \leq \lambda_2 \leq \dots $.
 Furthermore, the spectrum can be obtain via the Min-Max Principle. That is, the BF-Laplacian has a naturally associated energy functional defined by
\begin{equation*}
 E^F(f) := \frac{n}{\voleucl \left(\S^{n-1}\right) } \int_{HM} \left|L_X\left(\pi^{\ast}f \right)\right|^2 \ada.
\end{equation*}
The \emph{Rayleigh quotient} for $F$ is
\begin{equation*}
 R^F(f) := \frac{E^F(f)}{\int_M f^2\, \Omega^F}.
\end{equation*}

And the Min-Max principle says that the spectrum of the BF-Laplacian is given by
\begin{equation} \label{eq:min-max}
 \lambda_k = \inf_{V_k} \sup \left\{ R^F(f) \mid f \in V_k \right\}
\end{equation}
where $V_k$ runs over all the $(k+1)$-dimensional subspaces of $H^1(M)$ (the space of functions with derivatives in $L^2$).

When the manifold is not compact, the spectrum of the BF-Laplacian is in general not discrete, but the infimum of the spectrum, that we also denote by $\lambda_1$, is still obtained as the infimum of the Rayleigh quotient of functions in $H^1(M)$. So, depending on the context (compact or non-compact), $\lambda_1$ will refer to slightly different objects, but we hope that this will not cause too much confusion.

The BF-Laplacian is elliptic and symmetric with respect to the Holmes--Thompson volume $\Omega^F$, and, as such, is a weighted Laplacian. We denote by $\sigma^F$ the symbol metric of $\Delta^F$. Note that $\sigma^F$ is a dual Riemannian metric. If we identify $HM$ with the unit tangent bundle $S^FM$ of the metric $F$, and denote again by $\alpha^F$ the image of the angle measure on $S_x^FM$, we have, for $l_1,l_2 \in T^*_xM$,
\begin{equation*}
 \langle l_1,l_2 \rangle_{\sigma^F} = \frac{n}{\voleucl(\S^{n-1})}\int_{v \in S^F_xM} l_1(v)l_2(v) \alpha^F.
\end{equation*}
And another way of writing the energy of $\Delta^F$ is
\begin{equation*}
 E^F(f) =\int_{M} \lVert df \rVert_{\sigma^F}^2 \Omega^F.
\end{equation*}

\section{Adding a one-form increase the symbol} \label{sec:Randers_increase_symbol}

Proposition \ref{prop:spectrum_of_Randers_Finsler} will be an easy consequence of the following remark, that when a reversible metric $\bar F$ is modified by adding a $1$-form $\beta$, then the symbol metric gets bigger. This fact was proved for Randers metrics by He and Zheng \cite{HeZheng}, we give here a coordinate-free proof, and for which we do not need for the reversible part to be Riemannian. The proof of Proposition \ref{prop:lambda_1_negatively_curved} is also based on this, but we will have to be much more precise.

\begin{proposition} \label{prop:symbol_increase}
 Let $\bar F$ be a reversible Finsler metric on a manifold $M$, and let $\beta$ be a $1$-form on $M$ such that $F^*(\beta) <1$. Let $F= \bar F +\beta$. If we denote by $\sigma$ and $\bar\sigma$ the symbol metrics of the BF-Laplacians of $F$ and $\bar F$ respectively, we have, for any $f \in C^1(M)$ and $x\in M$,
\begin{equation*}
 \lVert d_xf \rVert_{\sigma} \geq \lVert d_xf \rVert_{\bar\sigma},
\end{equation*}
with equality if and only if $d_xf =0$ or $\beta_x = 0$.
\end{proposition}

The way we are going to prove this proposition is by writing explicitly the symbol of $F$ with respect to $\bar F$ and play around with the fact that $\bar F$ is reversible. We start by expressing how the different objects associated to a reversible Finsler metric behave when we apply the flip map. The flip map is the map $s\colon TM \rightarrow TM$ defined by $s(x,v) = (x, -v)$. We will abuse notation and also refer to the flip map on $HM$ as $s$.
 In all the following, we denote by $A$, $X$, $\alpha$ and $\Omega$ the Hilbert form, geodesic vector field, angle form and Holmes-Thompson volume form of the metric $F$, and $\bar A$, $\bar X$, $\bar\alpha$ and $\bar\Omega$ the same objects for the reversible metric $\bar F$.

\begin{lemma} \label{lem:flip_action}
 Let $\bar F$ be a reversible metric, then
\begin{align*}
 s^*\bar A &= - \bar A, \\
 s_*X &= -X, \\
 s_* Y &= Y, \quad \text{for any } Y\colon HM \rightarrow VHM, \\
 s^* \bar \alpha &= (-1)^{n} \bar \alpha .
\end{align*}
So in particular, for any Borel set on $H_xM$, $U$, and any integrable function, $f$, if we fix an orientation for $H_xM$, we have
\begin{equation*}
 \int_{s(U)} f \, \bar \alpha = \int_U f\circ s \, \bar \alpha.
\end{equation*}
\end{lemma}

\begin{proof}
 Writing the definition of $\bar A$ and using the fact that $d\pi\circ ds  = d\pi$ (since $\pi \circ s = \pi$) directly gives that $s^* \bar A = -\bar A$. Now, using either that $\bar X$ is the Reeb field of $\bar A$ or that it is the generator of the geodesic flow of $\bar F$, one quickly deduces that $ds \circ X = - X \circ s$, hence $s_*X = -X$.

The equality $s_*Y= Y$ is immediate: in a local chart $ds \colon THM \rightarrow THM$ can be written as $ds(x, \xi; v, y) = (x, - \xi; v,-y)$ and a vector in $VHM$ has to be of the form $(x, \xi; 0, y)$.

Finally, since $s^*\bar A = - \bar A $, we have that 
\[
s^*(\bar A \wedge d\bar A^{n-1 })= s^*\bar A \wedge (ds^*\bar A)^{n-1 }= (-1)^{n} \bar A \wedge d\bar A^{n-1}.
\]
Using the definition of $\bar \alpha$ and the fact that $s^{\ast}\pi^*\bar \Omega = \bar \Omega$, we then deduce that $s^* \bar \alpha = (-1)^{n} \bar \alpha$. The last equation is just the change of variables formula.
\end{proof}

We can now express the Hilbert one-form, the angle and the symbol of $F$ with respect to $\bar F$. 
\begin{lemma} \label{lem:Finsler-Randers}
  Let $\bar F$ be a reversible Finsler metric on a manifold $M$, and let $\beta$ be a $1$-form on $M$ such that $F^*(\beta) <1$. Let $F= \bar F +\beta$. Then, we have, for some vector field $Y_0\colon HM \rightarrow VHM$
\begin{align*}
 X&= \frac{\bar X}{1 +\pi^*\beta (\bar X)} + Y_0,\\
 A &= \bar A + \pi^*\beta,\\
 \ada &= (1 +\pi^*\beta (\bar X)) \bar A \wedge d\bar A^{n-1},\\
 \Omega &= \bar\Omega ,\\
 \alpha &= (1 +\pi^*\beta (\bar X)) \bar \alpha.
\end{align*}
\end{lemma}
The proof of this lemma is the exact same as the proof of Proposition 3.1.1 in \cite{moi:these} (or Proposition 3 in \cite{BarthelmeColbois}). The only difference is that the mentioned results were given for Randers metric, i.e., when $\bar F$ is Riemannian, but that fact was never really used in the proof, we just needed $\bar F$ to be reversible. Note also that the fact that the Holmes-Thompson volume is left unchanged when adding a one-form to a reversible Finsler metric is not new and can be seen as a consequence of the Brunn--Minkowski inequality (see, for instance, \cite{BonnesenFenchel,Thompson_book}).
\begin{proof}
 Since both $X$ and $\bar X$ are geodesic vector fields there exists $m\colon HM \rightarrow \R$ and a vector field $Y\colon HM \rightarrow VHM$ such that $X = m\bar X + Y$ (see \cite{Fou:EquaDiff}).
Now, direct computations using the definition of $A$ given above and the fact that $\beta$ is linear yields  $A = \bar A + \pi^*\beta$.

Using that $X$ is the Reeb flow of $A$, $\bar X$ the Reeb flow of $\bar A$, and $A(Y)= \bar A (Y) = 0$ for any vertical vector field, i.e., for any $Y\colon HM \rightarrow VHM$, we get that $m = (1 +\pi^*\beta (\bar X))^{-1}$.

Since $A = \bar A + \pi^*\beta$, we get that $dA = d\bar A + \pi^*d\beta$. So, $dA^{n-1} = d\bar A^{n-1} + T$ where $T$ is a $(2n-2)$-form. Since $\pi^*d\beta$ is a $2$-form vanishing on $VHM$, and for any $Y_1, Y_2 \in VHM$, $i_{Y_1} i_{Y_2} d\bar A = 0$, $T$ can be given at most $n-2$ vertical vectors, i.e., if $Y_1, \dots, Y_{n-1} \in VHM$, then $i_{Y_1} \dots i_{Y_{n-1}} T = 0$. Now this implies that the top-form $A\wedge T$ vanishes, hence $\ada = (\bar A + \pi^*\beta)\wedge d\bar A^{n-1}$. 

Since $\ada$ and $\bar A \wedge d\bar A^{n-1}$ are both volume forms, there exists a function $\lambda$ such that $\ada = \lambda \bar A \wedge d\bar A^{n-1}$. Now,
\[
 i_{\bar X}(\ada) = (1+\pi^*\beta(\bar X)) d\bar A^{n-1} = \lambda d\bar A^{n-1},
\]
therefore $\lambda = 1+\pi^*\beta(\bar X)$.

Given our computation of $\lambda$, using the definition of $\Omega$ and $\bar\Omega$, yields
\begin{align*}
 \alpha &=\frac{\voleucl(\S^{n-1})}{\int_{H_xM} (1+\pi^*\beta(\bar X)) \bar \alpha} (1+\pi^*\beta(\bar X)) \bar \alpha, \\
 \Omega &= \frac{\int_{H_xM} (1+\pi^*\beta(\bar X)) \bar \alpha}{\voleucl(\S^{n-1})}  \bar\Omega.
\end{align*}
But, by Lemma \ref{lem:flip_action} and since $\pi \circ s = \pi$, we obtain
\begin{align*}
 \int_{H_xM} \pi^*\beta(\bar X) \bar \alpha &= \int_{s^{-1}(H_xM)} \pi^*\beta(\bar X)\circ s \bar \alpha 
   = \int_{H_xM} \pi^*\beta(\bar X \circ s) \bar \alpha \\
   &= \int_{H_xM} \pi^*\beta(-ds \circ \bar X ) \bar \alpha 
   = \int_{H_xM} \beta_\pi(-d\pi \circ ds \circ \bar X ) \bar \alpha \\
   &= \int_{H_xM} -\beta_\pi(d\pi \circ \bar X ) \bar \alpha 
   = - \int_{H_xM} \pi^\ast\beta(\bar X ) \bar \alpha.
\end{align*}
So $\int_{H_xM} \pi^*\beta(\bar X) \bar \alpha =0$, hence 
\begin{equation*}
 \alpha =(1+\pi^*\beta(\bar X)) \bar \alpha  \quad \text{and} \quad
 \Omega = \bar \Omega. \qedhere
\end{equation*}



\end{proof}

We are now ready to prove Proposition \ref{prop:symbol_increase}. We in fact give a more precise evaluation of the symbol since we will use it later.

\begin{lemma} \label{lem:symbol_increase_precise_statement}
  Let $\bar F$ be a reversible Finsler metric on a manifold $M$, and let $\beta$ be a $1$-form on $M$ such that $F^*(\beta) <1$. Let $F= \bar F +\beta$. Let $H^+_xM := \lbrace \xi \in H_xM \mid \pi^*\beta(\bar X) (x,\xi) \geq 0\rbrace$. Then, for any $f \in C^1(M)$ and $x\in M$,
\begin{align*}
 \lVert d_xf \rVert^2_{\sigma} &= \frac{2n}{\voleucl(\S^{n-1})} \int_{H^+_xM} \frac{(L_{\bar X} \pi^* f)^2}{1 - (\pi^*\beta(\bar X))^2} \bar \alpha \\
  &\geq \lVert d_xf \rVert^2_{\bar\sigma},
\end{align*}
and the equality is realized only when $d_xf =0$ or $\beta_x = 0$.
\end{lemma}
Note that $H^+_xM$ is just the image on $HM$ of the set of vectors in $TM$ where $\beta$ is non-negative. Note also that in this proof, we will use the fact that the angle form $\bar \alpha$ gives a Lebesgue measure on $H_xM$. This is always true when the Finsler metric is $C^2$, but fails when we define the angle for less regular metric. The angle for a $C^0$ Finsler metric can be defined as the pullback by the Legendre transform of the vertical part of the symplectic volume on $S^*M$, the co-tangent unit bundle.

\begin{proof}
The proof is just a simple rewriting of the symbols, using Lemmas \ref{lem:flip_action} and \ref{lem:Finsler-Randers}. First, by Lemma \ref{lem:Finsler-Randers}, since $L_{Y_0}\pi^* f =0$, we have
\begin{equation*}
 (L_{X} \pi^* f)^2 = \frac{(L_{\bar X} \pi^* f)^2}{(1 + \pi^*\beta(\bar X))^2} \quad \text{and} \quad \alpha = (1 + \pi^*\beta(\bar X)) \bar \alpha.
\end{equation*}
So,
\begin{equation*}
   \lVert d_xf \rVert^2_{\sigma} = \frac{n}{\voleucl(\S^{n-1})} \int_{H_xM} (L_{X} \pi^* f)^2\alpha = \frac{n}{\voleucl(\S^{n-1})} \int_{H_xM} \frac{(L_{\bar X} \pi^* f)^2}{1 + \pi^*\beta(\bar X)} \bar \alpha.
\end{equation*}

Now, setting $H^-_xM := \lbrace \xi \in H_xM \mid \pi^*\beta(\bar X) (x,\xi) \leq 0\rbrace$, we have that $H_xM = H^+_xM \cup H^-_xM$, $H^+_xM \cap H^-_xM$ is of measure zero except when $\beta_x =0$, and $s(H^+_xM)= H^-_xM$. So, using Lemma \ref{lem:flip_action}, we obtain
\begin{align*}
 \int_{H_xM} \frac{(L_{\bar X} \pi^* f)^2}{1 + \pi^*\beta(\bar X)} \bar \alpha &= \int_{H^+_xM} \frac{(L_{\bar X} \pi^* f)^2}{1 + \pi^*\beta(\bar X)} \bar \alpha + \int_{H^-_xM} \frac{(L_{\bar X} \pi^* f)^2}{1 + \pi^*\beta(\bar X)} \bar \alpha\\
   &=\int_{H^+_xM} \frac{(L_{\bar X} \pi^* f)^2}{1 + \pi^*\beta(\bar X)} \bar \alpha + \int_{H^+_xM} \frac{((L_{\bar X} \pi^* f )\circ s)^2}{1 + \pi^*\beta(\bar X) \circ s} \bar \alpha\\
   &=\int_{H^+_xM} \frac{(L_{\bar X} \pi^* f)^2}{1 + \pi^*\beta(\bar X)} + \frac{(- L_{\bar X} \pi^* f )^2}{1 - \pi^*\beta(\bar X)} \bar \alpha\\
   &=2\int_{H^+_xM} \frac{(L_{\bar X} \pi^* f)^2}{1 - (\pi^*\beta(\bar X))^2} \bar \alpha.
\end{align*}
Since $(\pi^*\beta(\bar X))^2$ is positive outside of the directions in the kernel of $\beta$, we get 
\begin{align*}
 \lVert d_xf \rVert^2_{\sigma} &= \frac{2n}{\voleucl(\S^{n-1})} \int_{H^+_xM} \frac{(L_{\bar X} \pi^* f)^2}{1 - (\pi^*\beta(\bar X))^2} \bar \alpha \\
  &\geq \lVert d_xf \rVert^2_{\bar\sigma},
\end{align*}
with equality if and only if $\beta_x$ is zero or $d_xf$ is zero.
\end{proof}

Proposition \ref{prop:spectrum_of_Randers_Finsler} is then an immediate corollary:
\begin{corollary} \label{cor:compact_case_bigger_spectrum}
  Let $\bar F$ be a reversible Finsler metric on a \emph{closed} manifold $M$, and let $\beta$ be a $1$-form on $M$ such that $F^*(\beta) <1$. Let $F= \bar F +\beta$. Then $\vol(M,F)=\vol(M,\bar F)$, and, if $\beta$ is not identically zero, then for all $k\geq 1$,
\[
 \lambda_k(F) > \lambda_k(\bar F).
\]
\end{corollary}
Note that this result is still true for a compact manifold with boundary, but we chose to restrict ourselves to the closed case for simplicity.

Notice also that, if one knows a priori that a Finsler metric $F$ is of the form $F =\bar F + \beta$ and that $\lambda_1(F) = \lambda_1(\bar F)$ then the above corollary implies that $\beta=0$ and hence $F=\bar F$.

The fact that the volume of $M$ with respect to the metric $F$ or $\bar F$ is unchanged was already proved in Lemma \ref{lem:Finsler-Randers}, but we recall it here to emphasize the fact that the increase in the spectrum is not obtained by just shrinking the volume. 

\begin{proof}
Let $f_0$ be a constant function on $M$, and $f_1, \dots, f_k$ be eigenfunctions for $\lambda_1(F), \dots \lambda_k(F)$. Let $V$ be the $(k+1)$-dimensional subspace generated by $f_0, \dots, f_k$. By the Min-Max principle (given by equation \eqref{eq:min-max} in Section \ref{sec:background}), we have,
\begin{multline*}
 \lambda_k(F) = \max_{f \in V} \left\{ R^F(f) \right\}
  = \max_{f \in V} \left\{ \frac{\int_M \rVert df\lVert_{\sigma}^2 \Omega}{\int_M f^2 \Omega} \right\}\\
  > \max_{f \in V} \left\{ \frac{\int_M \rVert df\lVert_{\bar\sigma}^2 \bar\Omega}{\int_M f^2 \bar \Omega} \right\} 
  \geq \inf_{V_k} \sup_{f \in V_k} \left\{ R^{\bar F} (f)\right\} = \lambda_k(\bar F),
\end{multline*}
where in the last inequality, $V_k$ runs over all the $(k+1)$-dimensional subspace of $H^1(M)$ (or $C^{\infty}(M)$). 

The reason that the first inequality above is strict, is that any non-constant eigenfunction cannot have a vanishing derivative of all order (\cite{DonFef,Lin}). Hence, for any $f \in V$, $f$ not a constant, the support of $df$ has to be everything, and therefore intersects the support of $\beta$ in a set of positive measure, except if $\beta$ is identically zero.
\end{proof}

To finish proving our claim about the examples of Theorem \ref{thm:same_length_spectrum_higher_spectrum} we also need the following result (which is not new, see for instance \cite{Fang:thermo,FangFoulon})
\begin{lemma} \label{lem:canonical_time_change}
Let $\bar F$ be a reversible Finsler metric on a manifold $M$, let $\beta$ be a $1$-form on $M$ such that $F^*(\beta) <1$, and let $F= \bar F +\beta$. We have
\begin{enumerate}
 \item The geodesic flow of $F$ is a time change of the flow of $\bar F$ if and only if the $1$-form $\beta$ is closed.
 \item If $\beta$ is exact, then the time change does not modify the lengths of closed geodesics. In fact, the two flows are smoothly conjugate. Moreover, if the geodesic flows are Anosov, then it is an equivalence, i.e., $F= \bar F +\beta$ and $\bar F$ have the same marked length spectrum if and only if $\beta$ is exact.
\end{enumerate} 
\end{lemma}
Note that the Anosov condition for the equivalence in the second part of the Lemma is probably not optimal, but we will not use that part of the Lemma anyway.

\begin{proof}
 Recall from Lemma \ref{lem:Finsler-Randers} that the geodesic vector fields $X$ and $\bar X$ are related by $X = m \bar X + Y_0$, for some vertical vector field $Y_0$ and $m\colon HM \rightarrow \R$ given by $m = (1 + \pi^*\beta(\bar X))^{-1}$. So the geodesic flow of $F$ is a time change of the flow of $\bar F$ if and only if $Y_0=0$. Since,
\[
 0 = i_XdA = m i_{\bar X} d\bar A + i_{Y_0}d\bar A +m i_{\bar X} d\pi^*\beta  + i_{Y_0} d\pi^*\beta = i_{Y_0}d\bar A +m i_{\bar X} \pi^*d\beta,
\]
we get that, if, $d\beta = 0$, then $i_{Y_0}d\bar A =0$, and hence $Y_0 = 0$ (since $d\bar A$ is a symplectic form that is already zero on $\bar X$). So if $\beta$ is closed, then the geodesic flows are time changes of one another.

Now, if $Y_0=0$, then $i_{\bar X} \pi^*d\beta =0$. For $(x,\xi) \in HM$, let $(x,v)\in TM$ be the vector in the direction of $\xi$ and such that $\bar F(x,v)=1$. Then we have
\[
0 = d\beta_x(d\pi(\bar X (x,\xi), \cdot) = d\beta_x(v, \cdot).
\]
Hence, $d\beta_x(v, \cdot)=0$ for any $(x,v)\in TM$, that is, $d\beta =0$.

We can now prove the second part of the Lemma. First suppose that $\beta$ is exact, then for any closed, $C^1$ curve $c\colon [0,1] \rightarrow M$, the length of $c$ for $F$ is
\[
 l_F(c)= \int_0^1 F(c(t),\dot{c}(t)) dt = \int_0^1 \bar F(c(t),\dot{c}(t)) + \beta_{c(t)}(\dot{c}(t)) dt = l_{\bar{F}}(c) + \int_c \beta = l_{\bar{F}}(c).
\]
So, when $\beta$ is exact, the length of any closed curve stays unchanged, hence the marked length spectrum stays the same. Moreover, when $\beta$ is exact, it is easy to see that it is a trivial time change (in the terminology of \cite{KatokHassel}), i.e., it is a time change that is also a smooth conjugation.

Now suppose that the geodesic flows are Anosov and that the marked length spectrum are equal. Since the flows are Anosov geodesic flows, they are transitive, so the periodic orbits are dense in $HM$. And since $F$ and $\bar F$ have same marked length spectrum, the computations above show that for any periodic geodesic $\gamma \subset M$, $\int_{\gamma} \beta = 0$. So in particular, if we denote by $\bar \varphi^t$ the geodesic flow of $\bar F$ on $HM$, we see that the cocycle $\psi \colon HM \times \R \rightarrow \R$ defined by
\[
 \psi((x,\xi), T) = \int_{\{\bar \varphi^t(x,\xi) \mid 0\leq t \leq T\}} \pi^*\beta,
\]
is zero on every periodic orbit. So applying Liv\v{s}ic Theorem shows that $\psi$ is a coboundary so $\beta$ is exact.
\end{proof}

We can finally give the
\begin{proof}[Proof of Theorem \ref{thm:same_length_spectrum_higher_spectrum}]
 Let $\beta$ be an exact $1$-form, not identically zero, and such that $\bar F^*(\beta)<1$. Let $F= \bar F + \beta$. Then, by Lemma \ref{lem:canonical_time_change}, $F$ and $\bar F$ have the same marked length spectrum, by Lemma \ref{lem:Finsler-Randers}, they have the same volume, and Corollary \ref{cor:compact_case_bigger_spectrum} implies that the spectrum of $F$ is strictly greater than the spectrum of $\bar F$.
\end{proof}

\section{Finsler-Randers metrics with no upper bound for $\lambda_1$} \label{sec:family_with_exploding_lambda1}

\subsection{Family with a fixed reversible part and a fixed one-form}

Looking again at Lemma \ref{lem:symbol_increase_precise_statement}, we can make the following easy observation: If $\bar F$ and $\beta$ are fixed and we set $F_t = \bar F + t\beta$, for $0\leq t < (\sup \bar F^*(\beta))^{-1}$, then we see that the symbols associated with $F_t$ are increasing with $t$. Indeed, for any smooth function $f$, we have
\begin{equation*}
 \lVert d_xf \rVert^2_{\sigma_t} = \frac{2n}{\voleucl(\S^{n-1})} \int_{H^+_xM} \frac{(L_{\bar X} \pi^* f)^2}{1 - t^2(\pi^*\beta(\bar X))^2} \bar \alpha
\end{equation*}

So applying the proof of Corollary \ref{cor:compact_case_bigger_spectrum} shows that the spectrum of $F_t$ is strictly increasing in $t$. So a question one might ask is:\\
 \emph{What is the limit as $t$ tends to $(\sup \bar F^*(\beta))^{-1}$ of $\lambda_1(F_t)$?}

We will not answer that question in full generality, but will concentrate instead on the infinite case:\\
 \emph{Does there exist examples of Finsler manifolds $(M,\bar F)$ and form $\beta$ such that the limit of $\lambda_1(F_t)$ is infinite?}

It turns out that the answer to that second question is yes. 

First, we note that, when $t$ tends to $(\sup \bar F^*(\beta))^{-1}$, then, in some places at least, the symbol for $F_t$ tends to explode. To see that, we rewrite one more time $ \lVert d_xf \rVert_{\sigma_t}$, but in a slightly different way. We write $\bar SM$ for the unit tangent bundle of $\bar F$, and $\bar S^+M$ for the projection of $H^+M$ on $\bar SM$. We will also denote by $\bar \alpha$ the angle of $\bar F$ on the unit tangent bundle. Finally, for any $(x,v) \in \bar S M$, we let $c_{x,v}(t)$ be the geodesic of $\bar F$ on $M$ through $x$ in the direction $v$. Then,
\begin{equation}
\begin{aligned}
 \lVert d_xf \rVert^2_{\sigma_t} &= \frac{2n}{\voleucl(\S^{n-1})} \int_{H^+_xM} \frac{(L_{\bar X} \pi^* f)^2}{1 - t^2(\pi^*\beta(\bar X))^2} \bar \alpha  \\
  &= \frac{2n}{\voleucl(\S^{n-1})} \int_{\bar S^+_xM} \frac{1}{1 - t^2(\beta_x(v))^2} \left(\left.\frac{d}{dt}f(c_{x,v}(t))\right|_{t=0}\right)^2 \bar \alpha. 
\end{aligned} \label{eq:formula_symbol}
\end{equation}
Now, recall that $\bar F^*(\beta_x)= \sup \{ \beta (v) \mid v \in \bar S_xM \}$. Let $v_x\in\bar S_xM$ be the vector such that $\beta (v_x) = \bar F^*(\beta_x)$. 

So, when $t$ tends to $\left(\sup_x \bar F^*(\beta_x)\right)^{-1}$, we see that the norm $\lVert d_xf \rVert^2_{\sigma_t}$ stays bounded if and only if $\bar F^*(\beta_x) < \sup_x \bar F^*(\beta_x)$ or $\frac{d}{dt} f(c_{x,v_x}(t))|_{t=0} = 0$. 
 
Note that if we write $\mathcal{L}_{\bar F} \colon TM \rightarrow T^*M$ for the Legendre transform of $\bar F$, then by definition $v_x = \mathcal{L_{\bar F}}^{-1}(\beta_x)$ (see for instance \cite{moi:these} for the definition and some basic facts about the Legendre transform).

Moreover, we can rewrite $\frac{d}{dt} f(c_{x,v_x}(t))|_{t=0}$ as 
\[
 \frac{d}{dt} f(c_{x,v_x}(t))|_{t=0} = d_xf(v_x) = L_{\mathcal{L_{\bar F}}^{-1}(\beta)} f (x).
\]

So all we have to do to find a one parameter family of Finsler metric $F_t = \bar F +t \beta$ such that $\lambda_1(F_t)$ tends to infinity, is to choose a one form $\beta$ such that $\bar F^*(\beta)$ is constant on $M$ and such that, if $f$ is a smooth function invariant by $\mathcal{L_{\bar F}}^{-1}(\beta)$, i.e., such that $L_{\mathcal{L_{\bar F}}^{-1}(\beta)} f= 0$, then $f$ is constant. 

One easy way of making sure that this second point is verified is by choosing the one form $\beta$ such that the flow of the vector field $\mathcal{L_{\bar F}}^{-1}(\beta)$ admits a dense orbit. Indeed, $f$ would then be constant on a dense orbit, hence constant everywhere.

Hence we proved
\begin{proposition}
 Let $\bar F$ be a reversible Finsler metric on a closed manifold $M$. Suppose that there exists $\beta$, a $1$-form on $M$, such that:
\begin{itemize}
 \item $\bar F^*(\beta) = 1$;
 \item The vector field $\mathcal{L_{\bar F}}^{-1}(\beta)$ admits a dense orbit.
\end{itemize}
Then the one-parameter family of Finsler metrics $F_t = \bar F +t \beta$, $0\leq t <1$, is such that $\vol(M, F_t) = \vol(M,\bar F)$ and
\[
 \lim_{t\rightarrow 1} \lambda_1(F_t) = +\infty.
\]
\end{proposition}
Obviously, such a $\beta$ does not exist on every manifold $M$. In particular, the only surface that can support a $1$-form satisfying to the first point is the torus. But once the topological obstructions are taken care of, then it is very easy to construct such a $\beta$:
Start with a vector field $Z$ with no zeros and admitting a dense orbit, renormalize it so that $\bar F (Z) = 1$, then define $\beta:= \mathcal{L}_{\bar F}(Z)$.

\begin{rem}
 Note also that it is certainly not always the case that the limit of $\lambda_1(F_t)$ is infinite. Just take a one-form $\beta$ such that $\bar F^*(\beta)$ is not constant and take a function $f$ which support is included in the set $\{ x \in M \mid \bar F^*(\beta_x) < \sup \bar F^*(\beta) \}$, then the Rayleigh quotient $R^{F_t}(f)$ stays bounded in $t$, hence so does $\lambda_1(F_t)$.

Similarly, one can find a $\beta$ such that $\bar F^*(\beta)$ is constant but $\lambda_1(F_t)$ is still not infinite. Take for instance the $2$-torus $\mathbb{T}^2 = \R^2/\Z^2$ with $\bar F$ the flat Riemannian metric and $\beta = dx$, then it is easy to check that any function depending only on $y$ will have a bounded Rayleigh quotient.

So the above Proposition is, if not optimal, at least pretty close to being so.
\end{rem}

\subsection{Family with a fixed reversible part and a fixed length spectrum} \label{sec:fixed_length_exploding_lambda}

One of the downside of the previous proposition is that a form $\beta$ satisfying to the conditions given can be closed (take for instance $\beta = \cos \rho\, dx + \sin \rho\, dy$ on the flat torus $\mathbb{T}^2 = \R^2/\Z^2$, where $\pi \rho \notin \mathbb{Q}$), and therefore the geodesic flows of the metrics $F_t$ are time change of each others, but $\beta$ cannot be exact. So, unfortunately, in the examples obtained above, the marked length spectrum varies.

It is in fact not possible to come up with examples of a \emph{fixed} exact $1$-form $\beta$ such that $\lambda_1(F_t)$ tends to infinity. Indeed, since $\beta$ is exact, there must exist $x_0\in M$ such that $\beta_{x_0}=0$, hence there exists a neighborhood $U$ of $x_0$ such that the $\bar F$-norm of $\beta$ on $U$ is, say, at most half of the maximum of the $\bar F$-norm of $\beta$ on $M$. Hence any function $f$ with support in $U$ will have a Rayleigh quotient for $F_t$ bounded independently of $t$. So in particular $\lambda_1(F_t)$ can be big but not unbounded.

However, if we are willing to replace the family $F_t = \bar F + t\beta$, where $\beta$ is a fixed exact one-form by a family $F_{t,\eps} = \bar F + t\beta_{\eps}$, where $\beta_{\eps}$ is a family of exact one-forms, then we can obtain an infinite limit for $\lambda_1(F_{t,\eps})$.
From what we discussed above, one thing is clear: We will have to take a family of exact forms $\beta_{\eps}$ such that $\bar F^*(\beta_{\eps})$ tends to $1$ outside of a set of zero measure. Unfortunately, the vector fields $\mathcal{L}_{\bar F}^{-1}(\beta_{\eps})$ cannot admit a dense orbit, so we will have to work around that problem. I will just give an ad hoc construction on the $2$-torus with a flat metric. This construction could easily be extended to the $n$-torus and to any Riemannian metric, but cannot be extended to other manifolds as such. I did not pursue trying to find a general rule to obtain such metrics, but it would be surprising if the tori were the only manifolds admitting such examples.

Let $\T^2=\R^2/\Z^2$, $(x,y)$ be global coordinates on $\T^2$ and $g_0$ be the flat metric.

We are going to build $\beta_{\eps}$ as the differential of a certain function on $\mathbb{T}^2$.

Let $f_0 \colon S^1 = \R/\Z \rightarrow S^1 $ be the function defined by 
\begin{equation*}
 f_0(t)=\begin{cases}
                t &\text{if }  \,0\leq t\leq 1/2\\
                -t + 1 &\text{if } \,1/2\leq t\leq 1\\
        \end{cases}
\end{equation*}
Now, for any $\eps>0$, let $f_{\eps} \colon S^1 \rightarrow S^1$ be a smooth approximation of $f_0$ such that $f_{\eps}(0) = f_0(0)$, $f_{\eps}(1/2) = f_0(1/2)$, $|f'_{\eps}(t)| = 1$ for any $t \in [\eps, 1/2 -\eps] \cup [1/2 +\eps, 1-\eps]$, and $|f'_{\eps}(t)| \leq 1$ for any $t \in [-\eps,\eps] \cup [1/2 -\eps, 1/2 + \eps]$

Let $\rho \in \R$ such that $\pi\rho \notin \mathbb{Q}$. We define
\[
 h_{\eps}(x,y) := \cos \rho f_{\eps}(x) + \sin \rho f_{\eps}(y),
\]
and set
\begin{equation*}
 \beta_{\eps}:= dh_{\eps}.
\end{equation*}
Let $A_{\eps}^1$, $A_{\eps}^2$, $B_{\eps}^1$ and $B_{\eps}^2$ be the annuli given by
\begin{align*}
 A_{\eps}^1 &= \{ (x,y)\in \T^2 \mid -\eps < x < \eps \}, &  A_{\eps}^2 &= \{ (x,y)\in \T^2 \mid 1/2 -\eps < x < 1/2+\eps \}, \\
B_{\eps}^1 &= \{ (x,y)\in \T^2 \mid -\eps < y < \eps \}, &  B_{\eps}^2 &= \{ (x,y)\in \T^2 \mid 1/2 -\eps < y < 1/2+\eps \}.
\end{align*}
We set $C_{\eps} := \T^2 \smallsetminus (A_{\eps}^1\cup A_{\eps}^2\cup B_{\eps}^1 \cup B_{\eps}^2)$. On $C_{\eps}$, the norm of $\beta_{\eps}$ is $1$, and outside of $C_{\eps}$, it is less than $1$.
Note also for further reference that, in $C_{\eps}$, the vector field $\mathcal{L}_{g_0}^{-1}(\beta_{\eps}) =  \nabla h_{\eps}$ is of norm $1$ and points in the direction given by the angle $\rho$, or $\rho + \pi$ depending on which connected component of $C_{\eps}$ we consider).

Now that we have $\beta_{\eps}$, we can define the following Randers metric on $\T^2$
\begin{equation*}
 F_{t,\eps} = \sqrt{g_0}+ t\beta_{\eps}.
\end{equation*}

This family of Randers metrics verifies
\begin{proposition}
 Let $F_{t,\eps}$ be defined as above. Then, for all $\eps >0$ and all $0\leq t <1$, we have
\begin{itemize}
 \item $\vol(\T^2, F_{t,\eps}) = 1$;
 \item The marked length spectrum of $F_{t,\eps}$ is the marked length spectrum of the flat torus;
 \item But
\[
\lim_{(t,\eps)\rightarrow (1,0)} \lambda_1(F_{t,\eps}) = +\infty.
\]
\end{itemize}
\end{proposition}

\begin{proof}
 The first two points of the proposition were proven in Lemma \ref{lem:Finsler-Randers} and Lemma \ref{lem:canonical_time_change} respectively. All we have to do is prove the last.

In order to prove that third point, we are going to show that for any non constant smooth function $f$, with a fixed $L^2$-norm, the energy of $f$ for the metric $F_{t,\eps}$ tends to infinity as $t$ tends to $1$ and $\eps$ tends to $0$.

We set $\bar F = \sqrt{g_0}$ and use our previous notations. Let $f$ be a smooth function on $\T^2$ such that $\int_M f^2 \bar \Omega = 1$. Then, the Rayleigh quotient of $f$ is
\begin{equation*}
 R^{F_{t,\eps}} (f) = E^{F_{t,\eps}} (f) = \int_{\T^2} \lVert df \rVert^2_{\sigma_{t,\eps}} \bar \Omega \geq \int_{C_{\eps}} \lVert df \rVert^2_{\sigma_{t,\eps}} \bar \Omega,
\end{equation*}
where $C_{\eps}= \T^2 \smallsetminus (A_{\eps}^1\cup A_{\eps}^2\cup B_{\eps}^1 \cup B_{\eps}^2)$ is the set defined previously, on which $dh_{\eps}$ is of norm $1$.
Since
\[
 \lVert d_xf \rVert^2_{\sigma_{t,\eps}} = \frac{2}{\pi} \int_{\bar S^+_x\T^2} \frac{1}{1 - t^2(dh_{\eps}(x,v))^2} \left(d_xf(v)\right)^2 \bar \alpha,
\]
we deduce as before that, if there exists $x \in C_{\eps}$ such that 
\[
L_{\nabla h_{\eps}}f(x)= d_xf(\nabla h_{\eps}) \neq 0, 
\]
then $\lVert df \rVert^2_{\sigma_{t,\eps}}$ tends to infinity as $t$ tends to $1$ at $x$ and also in a small neighborhood of $x$. This implies that $R^{F_{t,\eps}} (f)$ tends to infinity as $t$ tends to $1$.

So all we are left to deal with are functions such that $L_{\nabla h_{\eps}}f(x)= 0$ for all $x\in C_{\eps}$ and all $\eps>0$. We are going to prove that such a function has to be constant, and this will prove our claim.

On $C_{\eps}$, $\nabla h_{\eps} = \pm V_{\rho}$, where $V_\rho\colon \T^2 \rightarrow S\T^2$ is the unit vector field pointing in the $\rho$ direction, i.e., $V_\rho = \cos \rho \frac{\partial }{\partial x} +\sin \rho \frac{\partial }{\partial y}$. The plus or minus sign depends on which of the four pieces of $C_{\eps}$ we are considering. So if $L_{\nabla h_{\eps}}f(x)= 0$ for all $x\in C_{\eps}$ and all $\eps>0$, then $L_{V_\rho}f(x)= 0$ for all $x\in C_{\eps}$ and all $\eps>0$. This implies that $f$ has to be constant on $C_{\eps}$ along the orbits of $V_{\rho}$. But since this has to be true for all $\eps>0$ and $C_\eps$ tends to the torus $\T^2$ minus four lines (the lines $x=0$, $x=1/2$, $y=0$ and $y=1/2$), by continuity of $f$, we deduce that $f$ has to be constant along the full orbits of $V_\rho$. Since $V_\rho$ has dense orbits, $f$ is constant.

In conclusion, if $f$ is not constant, then there exists $\eps>0$ and $x\in C_{\eps}$ such that $L_{\nabla h_{\eps}}f(x)\neq 0$. Hence, for any non constant function $f$,
\[
 \lim_{(t,\eps)\rightarrow (1,0)} R^{F_{t,\eps}}(f) = +\infty. \qedhere
\]
\end{proof}

\section{Negatively curved metrics, bottom of the spectrum and topological entropy} \label{sec:negatively_curved}

We will now switch our setting a bit. Let $M$ be a closed manifold and $\wt M$ be its universal cover. For any Finsler metric $F$ on $M$, we consider its lift $\wt F$ to $\wt M$ and we will be interested in $\lambda_1(\wt F)$ defined as the bottom of the $L^2$-spectrum of the BF-Laplacian $\Delta^{\wt F}$ on $\wt M$. That is,
\begin{equation*}
 \lambda_1(\wt F) = \inf \{ R^{F}(f) \mid f \in  C^{\infty}(\wt M) \cap L^2(\wt M, \Omega^{\wt F}) \}.
\end{equation*}

We will just stress once more that the $\lambda_1(\wt F)$ that we consider now is very different from the $\lambda_1(F)$ that we considered in the first part of this article. Among the differences let us mention two major ones: $\lambda_1(\wt F)$ is in general not an eigenvalue, but just the bottom of the spectrum, and $\lambda_1(\wt F)$ can be zero while $\lambda_1(F)$ is defined to be the first non-zero eigenvalue.

A classical inequality in Riemannian geometry is that $4\lambda_1(\wt g) \leq h^2$, where $h$ is the topological entropy of the geodesic flow. In fact, the classical proof of this inequality gives $4\lambda_1(\wt g) \leq v^2$, where $v$ is the volume entropy of the metric, but, by a famous result of Manning \cite{Manning},  $v\leq h$ and if $g$ has non-positive sectional curvature, then $v=h$.

The Riemannian proof adapted to the BF-Laplacian immediately yields that $4\lambda_1(\wt F) \leq n h^2$, where $h$ is still the topological entropy of the flow and $n$ is the dimension of the manifold. We will now recall that result (see Proposition \ref{prop:weak_inequality}) and construct examples showing that the stronger Riemannian inequality is not always satisfied, and hence proving Theorem \ref{thm:lambda_1_bigger_h}.

\subsection{Volume entropy for non-reversible Finsler metrics}

Before giving the proof of the weaker Finslerian inequality and the counter-examples to the Riemannian version, we need to precise what we mean by volume entropy in the non-reversible setting.
Let $F$ be a (non-reversible) Finsler metric on an open manifold $V$. If $B$ is a measurable set in $V$, we write $\vol(B,F) = \int_B \Omega^F$. We write $\vol(B) = \vol(B,F)$ if the Finsler metric we are using is clear from the context.

Since the distance associated with $F$ is not necessarily symmetric, there are two possible ways of defining the volume entropy: We can consider the rate of growth of the volume of \emph{forward balls}, or the rate of growth of \emph{backward balls}. A forward ball of radius $r$ is defined as
\[
 B^+(x,r):= \{ y \in V \mid d(x,y)\leq r\},
\]
while a backward ball of radius $r$ is given by
\[
 B^-(x,r):= \{ y \in V \mid d(y,x)\leq r\}.
\]
Then the \emph{forward volume entropy} is defined as
\[
 v^+(F) := \limsup_{R\rightarrow +\infty} \frac{1}{R} \log \vol(B^+(x,R)),
\]
and the \emph{backward volume entropy} is defined as
\[
 v^-(F) := \limsup_{R\rightarrow +\infty} \frac{1}{R} \log \vol(B^-(x,R)).
\]
For a generic non-reversible Finsler metric, these two volume entropies have no reasons to be equal. For instance, if one considers the Funk metric on a convex domain in $\R^n$ (see for instance \cite{HandbookHilbert}), then it is easy to see that $v^-(F)= +\infty$ while $v^+(F)$ is bounded. However, here are a few remarks that one can easily make about these objects:
\begin{itemize}
 \item The forward volume entropy of $F$ is equal to the backward volume entropy of the reversed metric $F\circ s$, i.e., $v^+(F)= v^-(F\circ s)$.
 \item If $F$ is quasi-reversible, i.e., if $C_F= \sup \{ F(x,-v) \mid F(x,v) = 1 \} <+\infty$, then 
\[
 \frac{1}{C_F} v^+(F) \leq v^-(F) \leq C_F v^+(F).
\]
\end{itemize}

The proof of the second point follows from the fact that, if $F$ is quasi-reversible, then $B^+(x,r/C_F) \subset B^-(x,r) \subset B^+(x,C_F r)$. The proof of the first point is easy once we rephrase what the forward and backward balls are in terms of flow. Let $S^FV$ be the unit tangent bundle for $F$ over $V$. Then the forward ball of radius $r$ is obtained by flowing $S^F_xV$ for a time $r$ under the geodesic flow of $F$. On the other hand, the backward ball can be seen to be obtained by flowing $S^{F\circ s}_xV$, the unit ball of $F\circ s$ on $T_xV$, under the geodesic flow of $F\circ s$, for a time $r$. Hence the equality.

So, the natural notion to choose in order to have Manning's result is the forward volume entropy. Manning's result was already extended to closed reversible Finsler manifolds by Egloff \cite{Egloff:Dynamics_uniform_finsler} and can further be extended to the non-reversible case. When $M$ is a closed manifold, the forward volume entropy of $M$ is by definition the forward volume entropy of $\wt M$.
\begin{theorem}[Manning \cite{Manning}, Egloff \cite{Egloff:Dynamics_uniform_finsler}]
 If $M$ is a closed manifold equipped with a possibly non-reversible Finsler metric, then $h(F) \geq v^+(F)$. Moreover, if $F$ has non positive flag curvature, then $h(F) = v^+(F)$.
\end{theorem}
The justification that Egloff gives in \cite{Egloff:Dynamics_uniform_finsler} to show that Manning's proof in \cite{Manning} holds in the Finsler context is still true in the non-reversible case.

Since $F$ has non-positive flag curvature if and only if $F \circ s$ has non-positive flag curvature, we see that in the non-positively curved case, $v^+(F) =v^-(F)$ if and only if the topological entropy of the geodesic flow of $F$ is equal to the topological entropy of the geodesic flow of $F\circ s$.

In fact, without using Manning's theorem, we can easily prove
\begin{proposition} \label{prop:volume_entropy_and_topologial_entropy_of_canonical_time_change}
 Let $\bar F$ be a reversible Finsler metric and $\beta$ an exact one-form on a closed manifold $M$. Let $F = \bar F + \beta$. Then 
\[
 v^+(F) = v^-(F)= v(\bar F), \quad \text{and} \quad h(F)=h(\bar F).
\]
\end{proposition}
\begin{proof}
The proof that the two topological entropies coincides is trivial: As $\beta$ is exact, the two geodesic flows are obtained by a time change that does not change the length of periodic geodesics (see Lemma \ref{lem:canonical_time_change}). Since the topological entropy is obtained as the exponential growth rate of the periodic orbits, the entropies of the two flows must coincide. Let us now prove that the volume entropies are the same.

 Let $\wt M$ be the universal cover of $M$, $x\in \wt M$ and $r>0$. Since $M$ is compact, there exists $C\geq 1$ such that 
\begin{align*}
 &\frac{1}{C}\vol (B^+(x,r), \wt F) \leq  \sharp\{ \gamma \in \pi_1(M) \mid d(x,\gamma \cdot x) \leq r \} \leq C \vol (B^+(x,r), \wt F), \text{ and}\\
 &\frac{1}{C}\vol (B^-(x,r), \wt F) \leq  \sharp\{ \gamma \in \pi_1(M) \mid d(\gamma \cdot x, x) \leq r \} \leq C \vol (B^-(x,r), \wt F),   
\end{align*}
where $d(\cdot, \cdot)$ is the distance for $\wt F$.
Now, since $\beta$ is an exact form on $M$, we see that $d(x,\gamma \cdot x) = d(\gamma \cdot x, x)$ (because the length of any closed curve in $M$ is unchanged). So, 
\begin{multline*}
 \vol (B^+(x,r), \wt F) \leq C \sharp\{ \gamma \in \pi_1(M) \mid d(x,\gamma \cdot x) \leq r \} \\ = C \sharp\{ \gamma \in \pi_1(M) \mid d(\gamma \cdot x, x) \leq r \} \leq C^2 \vol (B^-(x,r), \wt F).
\end{multline*}
Hence, $v^+(F) \leq v^-(F)$, and by symmetry we also have $v^-(F) \leq v^+(F)$.

Denoting by $\bar d(\cdot, \cdot)$ the distance for $\bar F$, since for all $\gamma \in \pi_1(M)$, $d(x,\gamma \cdot x) = \bar d(x,\gamma \cdot x)$, we also have $v^+(F) =v(\bar F)$.
\end{proof}

\subsection{The weak topological entropy inequality and a counterexample to the sharp version}

Let us quickly recall why, just by following the Riemannian proof, we obtain a weak inequality for $\lambda_1(\wt F)$. Note that for the following proposition, we do not need $\wt M$ to be the universal cover of a closed manifold. The result holds for any open manifold.

\begin{proposition} \label{prop:weak_inequality}
 Let $\wt F$ be a Finsler metric on a manifold $\wt M$ of dimension $n$. Then the bottom of the $L^2$-spectrum of the BF-Laplacian satisfies
\[
 \lambda_1(\wt F) \leq \frac{n}{4} \min\{(v^+(\wt F))^2, (v^-(\wt F))^2\}
\] 
\end{proposition}

\begin{proof}
 We fix a base point $O \in \wt M$ and define the forward and backward distance functions by, respectively, 
\[
 \rho^+(x) = d(O,x) \quad \text{and} \quad \rho^-(x) = d(x,O).
\]
By definition of $v^+(\wt F)$ and $v^-(\wt F)$, we have that for all $2s> v^+(\wt F)$ and $2t> v^-(\wt F)$, $e^{-s\rho^+(x)} \in L^2(\wt M)$ and $e^{-t\rho^-(x)} \in L^2(\wt M)$.

We will just give an upper bound for the Rayleigh quotient of $e^{-s\rho^+(x)}$. The case of $e^{-t\rho^-(x)}$ is exactly the same.

First, we have that
\begin{equation*}
L_X \pi^{\ast} e^{-s\rho^+} (x,\xi) = -s \left(L_X\pi^{\ast} \rho\right)(x,\xi) e^{-s\rho^+(x)}.
\end{equation*}
So, 
\begin{equation*}
 \int_{H\M} \left(L_X \pi^{\ast} e^{-s\rho^+} \right)^2 \ada = \int_{x\in \M} s^2 \left(\int_{\xi \in H_x\M} \left(L_X\pi^{\ast} \rho^+ (x,\xi)\right)^2 \alpha \right) e^{-2s\rho^+(x)} \Omega.
\end{equation*}
Now, $\left|L_X\pi^{\ast} \rho^+ (x,\xi) \right|\leq \left| d_x\rho (v_x) \right| = 1$, where $v_x$ is the vector in $S_x\wt M$ such that the direction, at $x$, of the geodesic from $O$ to $x$ is $v_x$. Therefore, we get
\begin{align*}
  \lambda_1(\wt F) &\leq \frac{n}{\voleucl\left(\S^{n-1}\right)} \left(s^2 \int_{\M} e^{-2s\rho^+(x)}  \int_{H_{x}\M} \alpha\, \Omega\right) \left(\int_{\M} e^{-2s\rho^+(x)} \Omega \right)^{-1} \\
 &\leq n s^2 . 
\end{align*}
Since the above inequality is true for all $s>v^+(\wt F)/2$, we deduce that $\lambda_1(\wt F)  \leq n (v^+(\wt F))^2/4$. Doing the same computation with $e^{-t\rho^-(x)}$ yields $\lambda_1(\wt F)  \leq n (v^-(\wt F))^2/4$. 
\end{proof}

We can now finish the construction of the last surprise: Finsler metrics such that the sharp inequality for the bottom of the spectrum is not verified. Let us recall Theorem \ref{thm:lambda_1_bigger_h}:
\begin{theorem}
 Let $g$ be a hyperbolic metric on a manifold $M$. Let $h$ be a smooth function on $M$ such that the zeros of $dh$ are isolated and $\lVert dh \rVert_{g^*} <1$. Let $F = \sqrt{g} + dh$, $\lambda_1(\wt F)$ be the bottom of the $L^2$-spectrum and $h(F)$ be the topological entropy of the geodesic flow of $F$. Then $F$ and $g$ have the same marked length spectrum and
\[
 4\lambda_1(\wt F) > h(F)^2 = (n-1)^2.
\]
\end{theorem}

The proof of this result also contains the proof of Proposition \ref{prop:lambda_1_negatively_curved}, one just has to make obvious notational changes. We hence do not provide an explicit proof of that proposition.
\begin{proof}
 The fact that $F$ and $g$ have the same marked length spectrum follows from Lemma \ref{lem:canonical_time_change}. Moreover, by Proposition \ref{prop:volume_entropy_and_topologial_entropy_of_canonical_time_change}, the topological entropy of $F$ is the same as the topological entropy of $g$ and they are equal to the volume entropy. And since $g$ is hyperbolic, $4\lambda_1(\wt g) = (n-1)^2 = h(g)^2$. 

So all we have to do to prove the above inequality is show that $\lambda_1(\wt F) > \lambda_1(\wt g)$. This would be trivial if $\lambda_1(\wt F)$ was an eigenvalue (the same argument as in the proof of Corollary \ref{cor:compact_case_bigger_spectrum} would immediately yields the answer), but this is in general not the case. Hence we have to work more.

Let us abuse notations a bit and write $h$ again for the lift of $h$ to the universal cover $\wt M$. The function $h$ will only be on the universal cover for the rest of this proof, so hopefully this will not cause any confusion.

With this abuse of notation, by our previous computations (see Equation \eqref{eq:formula_symbol}), we have that for any $f \in C^{\infty}(\wt M) \cap L^2(\wt M)$
\begin{equation*}
 \lVert df \rVert^2_{\sigma^{F}}  = \frac{2n}{\voleucl(\S^{n-1})} \int_{\bar S^+_x\wt{M}} \frac{1}{1 - (d_xh(v))^2} \left(d_xf(v)\right)^2 \bar \alpha,
\end{equation*}
where $\bar \alpha$ is just the Riemannian angle measure on the Riemannian unit spheres $\bar S_x\wt{M}$. So in particular,
\begin{align*}
 \lVert d_xf \rVert^2_{\sigma^F} &\geq \frac{2n}{\voleucl(\S^{n-1})} \left( \int_{\bar S^+_x\wt{M}} \left(d_xf(v)\right)^2 \bar \alpha + \int_{\bar S^+_x\wt{M}} (d_xh(v))^2 \left(d_xf(v)\right)^2 \bar \alpha \right) \\
&\geq \lVert d_xf \rVert^2_{g^*} + \frac{2n}{\voleucl(\S^{n-1})} \int_{\bar S^+_x\wt{M}} (d_xh(v))^2 \left(d_xf(v)\right)^2 \bar \alpha
\end{align*}

In the rest of the proof, we will be using the fact that $F$ is a Randers metric, i.e., that we have a Riemannian metric $g$ to work with. The fact that the bottom of the spectrum of $\wt F$ is strictly greater than the one for $\wt g$ should certainly still hold if $g$ was replaced by any reversible metric $\bar F$, but the proof would probably be more tedious (or at least we did not find an easy proof).

To obtain the strict inequality we are aiming for, we will first find a lower bound for $\lVert d_xf \rVert^2_{\sigma^F}$ in terms of $\lVert d_xf \rVert^2_{g^*}$.
We first notice that $\int_{\bar S^+_x\wt{M}} (d_xf(v))^2 \left(d_xh(v)\right)^2 \bar \alpha$ is minimized when $\nabla f (x)$ and $\nabla h(x)$ are orthogonal (here $\nabla$ and orthogonal are defined with respect to the hyperbolic metric $g$). So we suppose that $\nabla f (x)$ and $\nabla h(x)$ are orthogonal. We choose a coordinate system on $S_x\wt{M}$ such that for $v\in S_x\wt{M}$, $\theta(v)$ represents the angle between the direction of $\nabla h (x)$ and the projection of $v$ onto the plane containing $\nabla f (x)$ and $\nabla h(x)$.
In other words, we write
\[
 S_x\wt{M} =\left\{ (\theta, \xi) \in [-\pi,\pi] \times H^{n-2} \right\},
\]
where $H^{n-2}$ is the unit hemisphere of dimension $n-2$, and such that for $v= (\theta, \xi) \in S_x\wt{M}$ we have
\begin{align*}
 d_xh(v) &= \lVert d_xh\rVert_{g^*} \cos \theta \\
 d_xf(v) &= \lVert d_xf\rVert_{g^*} \sin \theta.
\end{align*}
If we write $d\theta d\xi$ for the Riemannian angle $\bar \alpha$ in the coordinates that we choose on $S_x\wt{M} = [-\pi,\pi] \times H^{n-2}$, then we have
\begin{align*}
 2 \int_{\bar S^+_x\wt{M}} (d_xf(v))^2 \left(d_xh(v)\right)^2 \bar \alpha &= \int_{\bar S_x\wt{M}} (d_xf(v))^2 \left(d_xh(v)\right)^2 \bar \alpha \\
&= \int_{H^{n-2}} \left( \int_{-\pi}^{\pi} \lVert d_xf\rVert^2_{g^*}\lVert d_xh\rVert_{g^*}^2 \cos^2 \theta \sin^2 \theta d\theta \right) d\xi \\
&= \lVert d_xf\rVert^2_{g^*}\lVert d_xh\rVert_{g^*}^2 \frac{\pi}{4} \int_{H^{n-2}} d\xi \\
 &= \frac{\lVert d_xf\rVert^2_{g^*}\lVert d_xh\rVert_{g^*}^2 }{8} \voleucl(\S^{n-1}).
\end{align*} 
Hence,
\begin{equation*}
  \lVert d_xf \rVert^2_{\sigma^F} \geq \lVert d_xf \rVert^2_{g^*}\left(1 + n\frac{\lVert d_xh\rVert^2_{g^*}}{8} \right).
\end{equation*}

Therefore, given the characterization of $\lambda_1$ has the infimum of the Rayleigh quotient of smooth $L^2$ functions, we have 
\[
 \lambda_1(\wt F) \geq \inf_{f\in L^2(\wt M)} \frac{\int_{\wt M} \lVert d_xf \rVert^2_{g^*}\left(1 + n\lVert d_xh\rVert^2_{g^*}/8 \right)\Omega^g}{\int_{\wt M} f^2 \Omega^g}.
\]

Recall that we chose $h$ such that $dh$ had isolated zeros. We denote these zeros by $z_k$, $k\in \mathbb{N}$. Since the $z_k$ are isolated, for $\eps>0$ small enough, there exists a constant $c=c(\eps)>0$ such that outside of some balls $B_k = B(z_k,\eps)$, we have $\lVert d_xh\rVert^2_{g^*} > c$. Moreover, taking $\eps$ small enough, we can suppose that the $B_k$ are pairwise disjoint (the radius of the balls can be chosen uniform, while still having them pairwise disjoint because the $z_k$ are the lifts of the isolated zeros of $dh$ on the compact manifold $M$). Up to taking $\eps$ smaller still, we can suppose that the balls $B'_k = B(z_k,2\eps)$ are still pairwise disjoint.

We will first show that the infimum in the equation above cannot be attained by functions with support inside the balls $B'_k$.
Let $h_{\text{Cheeger}}(B'_k)$ be the Cheeger constant on $B'_k$, i.e., 
\[
 h_{\text{Cheeger}}(B'_k) = \inf \left\{\frac{\mathrm{Area}(\partial U)}{\vol(U)} \mid U \text{ open}, \bar U \subset B'_k \right\}
\]
Note that since all the $B'_k$ are balls of the same radius and that the metric is hyperbolic (and hence homogeneous), $h_{\text{Cheeger}}(B'_k)$ is independent of $k$. By the proof of the classical Cheeger inequality (see for instance \cite[p.~91]{SchoenYau}), we have that if $f$ is a function such that $\mathrm{supp}\, f \subset \cup B'_k$, then
\[
\int_{ B'_k} \lVert d_xf \rVert^2_{g^*} \Omega^g \geq \frac{h_{\text{Cheeger}}^2(B'_k)}{4} \int_{B'_k}  f^2 \Omega^g.
\]
Hence, since $h_{\text{Cheeger}}(B'_k)$ is independent of $k$,
\begin{multline*}
 \int_{\wt M} \lVert d_xf \rVert^2_{g^*} \Omega^g = \int_{\cup B'_k} \lVert d_xf \rVert^2_{g^*} \Omega^g \geq \frac{h_{\text{Cheeger}}^2(B'_k)}{4} \int_{\cup B'_k}  f^2\Omega^g \\ = \frac{h_{\text{Cheeger}}^2(B'_k)}{4}\int_{\wt M}  f^2\Omega^g.
\end{multline*}
Now the Cheeger constant on a very small hyperbolic ball $B'_k$ is close to the Cheeger constant of a small Euclidean ball, so for $\eps >0$ small enough, the Cheeger constant on $B'_k$ is approximately $2/\eps$.
In particular, taking $\eps$ small enough, we see that if $f$ is such that $\mathrm{supp}\, f \subset \cup B'_k$, then 
\begin{equation} \label{eq:big_Rayleigh_quotient}
  \frac{\int_{\wt M} \lVert d_xf \rVert^2_{g^*}\left(1 + n\lVert d_xh\rVert^2_{g^*}/8 \right)\Omega^g}{\int_{\wt M} f^2 \Omega^g}\geq R^g(f) \geq \frac{1}{4\eps^2} > 10 (n-1)^2,
\end{equation}
where $R^g(f) = \int_{\wt M} \lVert d_xf \rVert^2_{g^*}\Omega^g /\int_{\wt M} f^2 \Omega^g$ is the Rayleigh quotient for the hyperbolic metric $g$.
We fix such an $\eps$ once and for all.

Let $f_i$ be a sequence of functions in $L^2(\wt M)$ such that $R^F(f_i)$ converges to $\lambda_1(\wt F)$. We suppose furthermore that all the $f_i$ are normalized so that $\int_{\wt M} f_i \Omega^g=1$. Our goal is to show that 
\[
 \lambda_1(\wt F) \geq \liminf_{i\rightarrow \infty} \int_{\wt M} \lVert d_xf_i \rVert^2_{g^*}\left(1 + n\lVert d_xh\rVert^2_{g^*}/8 \right)\Omega^g > \lambda_1(\wt g).
\]

We proceed by contradiction: Suppose that 
\[
\liminf_{i\rightarrow \infty}  \int_{\wt M} \lVert d_xf_i \rVert^2_{g^*}\left(1 + n\lVert d_xh\rVert^2_{g^*}/8 \right)\Omega^g = \lambda_1(\wt g). 
\]
Then, up to passing to a subsequence, we have
\begin{equation*}
\lim_{i\rightarrow \infty} \int_{\wt M} \lVert d_xf_i \rVert^2_{g^*} \Omega^g  = \lambda_1(\wt g) \quad \text{and} \quad
 \lim_{i\rightarrow \infty} \int_{\wt M} \lVert d_xf_i \rVert^2_{g^*}\lVert d_xh\rVert^2_{g^*} \Omega^g  = 0.
\end{equation*}

Since $\lVert d_xh\rVert^2_{g^*} >c$ outside of the balls $B_k$ defined above, we have
\[
 \lim_{i\rightarrow \infty} \int_{\wt M \smallsetminus \cup B_k} \lVert d_xf_i \rVert^2_{g^*} \Omega^g  = 0
\]
Morally, this means that the functions $f_i$ tends to be almost constant outside of the balls $B_k$, and since they are in $L^2(\wt M)$ they have to be almost $0$ outside of the $B_k$. So their Rayleigh quotient has to be close to the Rayleigh quotient of a function with support in the $B_k$, but we proved before that the Rayleigh quotient of such function is very big, which gives us a contradiction.

Let us be more precise: Let $\eta >0$ be arbitrary. Let $I>0$ such that, for $i >I$, $\int_{\wt M \smallsetminus \cup B_k} \lVert d_xf_i \rVert^2_{g^*} \Omega^g  <\eta$. We only consider $i>I$ in the rest of the proof.

We define the functions $f'_i$ in the following way:
\begin{equation*}
 f'_i(x) = \begin{cases}
            f_i(x) & \text{if } x \in \cup B_k= \cup B(z_k, \eps) \\
            0 & \text{if } x \notin \cup B'_k= \cup B(z_k, 2\eps),
           \end{cases}
\end{equation*}
and we choose $f'_i(x)$ to decrease linearly (with the radius) between $\partial B_k$ and $\partial B'_k$. 

Since ${\int_{\wt M \smallsetminus \cup B_k} \lVert d_xf_i \rVert^2_{g^*} \Omega^g  <\eta}$ and $\int_{\wt M} f_i^2 \Omega^g =1$, there exists a constant $\eta_1 = \eta_1(\eta, \eps)>0$, depending on $\eta$ and $\eps$ such that 
\[
 |R^g(f_i) -R^g(f'_i)| <\eta_1.
\]
 Moreover, the constant $\eta_1$ goes to zero as $\eta$ goes to zero (and gets bigger as $\eps$ gets smaller, but we fixed $\eps$ before). So for some small enough $\eta$ (depending on $\eps$), we have $\eta_1 \leq 1$.

Hence, by Equation \eqref{eq:big_Rayleigh_quotient}, since $\mathrm{supp}\, f'_i \subset \cup B'_k$, we have
\begin{equation*}
 R^g(f_i) \geq R^g(f'_i) -1 > 10(n-1)^2 - 1 > (n-1)^2 = \lambda_1(\wt g).
\end{equation*}
This is in contradiction with the fact that $\displaystyle \lim_{i\rightarrow +\infty} R^g(f_i) = \lim_{i\rightarrow +\infty} \int_{\wt M} \lVert d_xf_i \rVert^2_{g^*} \Omega^g  = \lambda_1(\wt g)$, and that ends the proof.
\end{proof}

Let us finish this article by a remark. All along this article I have been claiming that it is difficult to get a strict inequality in Theorem \ref{thm:lambda_1_bigger_h}, and the proof that I give certainly is more involved than the proof of Theorem \ref{thm:same_length_spectrum_higher_spectrum}. However, I might just have not been clever enough. If that may very well be true, I would however like to point out the following example that should serve as a word of caution, even though it is not in a cocompact setting.
\begin{proposition}
 Let $(\Hyp^n,g)$ be the hyperbolic $n$-space. Let $\beta$ be a one-form on $\Hyp^n$ such that $\lVert \beta_x \rVert_{g^*}$ tends to zero as $x$ approaches $\partial \Hyp^n$ (and such that $\lVert \beta_x \rVert_{g^*}<1$). Let $F = \sqrt{g} + \beta$. Then,
\[
 \lambda_1(F) = \lambda_1(g) = (n-1)^2/4.
\]
\end{proposition}
\begin{proof}
 From Proposition \ref{prop:symbol_increase}, we directly obtain that $\lambda_1(F) \geq \lambda_1(g) = (n-1)^2/4$.
Now, since $\lVert \beta_x \rVert_{g^*}$ tends to zero as $x$ leaves every compact, the Finsler metric $F$ is \emph{asymptotically Riemannian} in the terminology of \cite{BCCV}. That is, for any $C>1$, there exists a compact $K$ such that, for any $x \in \Hyp^n \smallsetminus K$ and any $v \in T_x \Hyp^n$, 
\begin{equation*}
 C^{-1} \leq \frac{F(x,v)}{\sqrt{g(x,v)}} \leq C.
\end{equation*}

So, by \cite[Theorem 4.1]{BCCV}, we have $4\lambda_1(F) \leq (v^+(F))^2$. All there is left to do is to show that $v^+(F)$ corresponds to the volume entropy of the hyperbolic metric.

In the rest of the proof, we use $F$ as a subscript when referring to object defined by the distance $d_F$ for $F$ and $g$ when using the hyperbolic distance $d_g$.

Fix $x_0 \in \Hyp_n$ and $C>1$. Let $K$ be a compact such that, outside of $K$, the ratio between $F$ and $g_0$ is bounded by $C$. Since $K$ is compact, there exists $C_1$ (depending on $C$) and $C_2 \leq C_3$ (depending on $C_1$) such that $K \subset B_{g}(x_0, C_2) \subset B^+_F(x_0, C_1) \subset B_{g}(x_0, C_3)$ and $K \subset B_{g}(x_0, C_1)$, where $B^+_F(x_0, C_1)$ is the (forward) ball of radius $C_1$ for $F$ and $B_{g}(x_0, C_2)$ is the ball of radius $C_2$ for $g$.

Let $A^+_F(x_0, C_1, R)$ be the (forward) annulus of radii $C_1$ and $C_1 + R$ for $F$.
Let $y \in A^+_F(x_0, C_1, R)$.

Since $B_{g}(x_0, C_2) \subset B^+_F(x_0, C_1)$, we have that $d_g(x_0, y) \geq C_2$.
Now, Let $x_1$ be the point on $\partial B_F(x_0, C_1)$ such that $d_F(x_0,y) = d_F(x_0, x_1) + d_F(x_1,y)$. Since the geodesic between $x_1$ and $y$ is outside of the compact $K$, we have 
\begin{equation*}
 R +C_1 \geq d_F(x_0,y) = d_F(x_0, x_1) + d_F(x_1,y) \geq C_1 + \frac{1}{C} d_g(x_1, y).
\end{equation*}
So, using that $B^+_F(x_0, C_1) \subset B_{g}(x_0, C_3)$, we have
\begin{equation*}
 d_g(x_0,y) \leq d_g(x_0,x_1) + d_g(x_1,y) \leq C_3 + CR.
\end{equation*}

Hence, we showed that $A^+_F(x_0, C_1, R) \subset A_g(x_0, C_2, CR + C_3)$. Since we can also compute the volume entropy by taking the exponential growth of annuli instead of balls and that the volume for $F$ and for $g$ are the same, we obtain
\begin{align*}
 v^+(F)  &= \limsup_{R\rightarrow +\infty} \frac{1}{R} \log \vol(A^+_F(x_0, C_1,R)) \\
 &\leq \limsup_{R\rightarrow +\infty} \frac{1}{R} \log \vol(A_g(x_0, C_2,CR +C_3)) \\
 &\leq C (n-1).
\end{align*}

So $v^+(F) \leq C (n-1)$. Since this is true for any $C>1$, we obtain that $v^+(F)= n-1$, which proves that $4\lambda_1(F) = (n-1)^2$.
\end{proof}

\bibliographystyle{amsplain}
\bibliography{surprises}

\end{document}